\documentclass{amsart}
\usepackage{amsmath,amssymb,amsthm,amscd,xypic}
\usepackage{color}
\usepackage{enumerate}
\usepackage[T1]{fontenc} 
\usepackage{hyperref}

\title{Silting, cosilting, and extensions of commutative ring}

\author{Simion Breaz, Michal Hrbek, George Ciprian Modoi}
\address[Simion Breaz and George Ciprian Modoi]{Babe\c s--Bolyai University, Faculty of Mathematics and Computer Science \\
1, Mihail Kog\u alniceanu, 400084 Cluj--Napoca, Romania}
\address[Michal Hrbek] {Institute of Mathematics of the Czech Academy of Sciences, \v{Z}itn\'{a} 25, 115 67 Prague, Czech Republic}
\email[Simion Breaz]{bodo@math.ubbcluj.ro}
\email[George Ciprian Modoi]{cmodoi@math.ubbcluj.ro}
\email[Michal Hrbek]{hrbek@math.cas.cz}

\subjclass[2010]{} \keywords{}

\thanks{ }
\date{\today}

\renewcommand{\iff}{if and only if }

\newcommand{\Z}{\mathbb{Z}}
\newcommand{\Q}{\mathbb{Q}}

\DeclareMathOperator{\Hom}{Hom} 
\DeclareMathOperator{\RHom}{{\bf R}Hom}

\DeclareMathOperator{\Ext}{Ext} 
\DeclareMathOperator{\Ker}{Ker} 
\DeclareMathOperator{\Coker}{Coker}

\newcommand\lotimes{\otimes^{\mathbf L}}

\newcommand{\A}{\mathcal{A}}

\newcommand{\I}{\mathcal{I}}
\newcommand{\CL}{\mathcal{L}}
\newcommand{\CK}{\mathcal{K}}
\newcommand{\Y}{\mathcal{Y}}
\newcommand{\C}{\mathcal{C}}
\newcommand{\D}{\mathcal{D}}

\newcommand{\U}{\mathcal{U}}
\newcommand{\X}{\mathcal{X}}
\newcommand{\V}{\mathcal{V}}
\newcommand{\W}{\mathcal{W}}

\newcommand{\CS}{\mathcal{S}}
\newcommand{\CZ}{\mathcal{Z}}
\newcommand{\CE}{\mathcal{E}}
\newcommand{\bp}{\mathfrak{p}}
\newcommand{\bq}{\mathfrak{q}}

\newcommand{\CX}{\mathcal{X}}
\newcommand{\CY}{\mathcal{Y}}

\newcommand{\ModR}{\hbox{\rm Mod-}R}

\newcommand{\Qcoh}{\hbox{\rm Qcoh-}}

\newcommand{\op}{^\mathrm{op}}
\newcommand{\Mod}[1]{\hbox{\rm Mod-}{#1}}
\newcommand{\add}{\mathrm{add}}
\newcommand{\Add}{\mathrm{Add}}

\newcommand{\Prod}{\mathrm{Prod}}

\newcommand{\Spec}[1]{\mathrm{Spec}({#1})}
\newcommand{\Dr}{\mathbf{D}}
\newcommand{\Der}[1]{\mathbf{D}({#1})}

\newcommand{\Loc}{\mathrm{Loc}}

\newcommand{\susp}{\mathrm{susp}\,}
\newcommand{\cosusp}{\mathrm{cosusp}\,}

\newcommand{\hocolim}{\mathrm{hocolim}}

\theoremstyle{plain}
\newtheorem{thm}{Theorem}[section]
\newtheorem{lemma}[thm]{Lemma}
\newtheorem{prop}[thm]{Proposition}
\newtheorem{cor}[thm]{Corollary}

\newtheorem{thm-intro}{Theorem}

\theoremstyle{definition}

\theoremstyle{remark}
\newtheorem{rem}[thm]{Remark}

\newtheorem{setting}{Setting}

\begin{document}


\begin{abstract}
We study the transfer of (co)silting objects in derived categories of module categories via the extension functors induced by a morphism of commutative rings. It is proved that the extension functors preserve (co)silting objects of (co)finite type. In many cases the bounded silting property descends along faithfully flat ring extensions. In particular, the notion of bounded silting complex is Zariski local.
\end{abstract}



\maketitle

\section{Introduction}

Silting theory  provides useful tools in the study of various triangulated categories. We refer to \cite{An-19} and \cite{NSZ} for details about the influence of this. In particular, it studies $t$-structures with special properties that provide good approximations.  In the case of the unbounded derived category of a module category these $t$-structures are induced by some objects (called silting, respectively cosilting, objects) that can be interpreted as generalizations of tilting complexes (in particular, they also generalize the $n$-tilting modules), \cite{AMV:2015},  \cite{Wei-isr}. In \cite{An-Hr} it is shown that if we are in the derived category of a commutative ring, the silting complexes of length $2$ are related with other structures associated to that ring (Gabriel filters, torsion theories of finite type, the spectrum, cf. \cite{Ga-Pr}. Moreover, it was proved in \cite{AJS}, \cite{AH21}, \cite{H20} and  \cite{HHZ21} that there exists a strong connection between the spectrum of the ring and the $t$-structures associated to (co)silting objects in the unbounded derived category. 

Since we have such a connection, it is natural to ask if the silting property is Zariski local or, even more, if it is an ad-property (see \S \ref{subsec-locality} for the relevant definitions). It is already known that the properties of being projective,  $1$-tilting and $2$-silting are ad-properties (cf. \cite{RG}, \cite{HST20}, and \cite{Br20}). The Zariski  locality for $n$-tilting modules was also proved in \cite{HST20}. In order to provide answers for this question, we will study transfers, induced by a morphism $\lambda:R\to S$ of commutative rings, of the (co)silting properties of objects from $\Der R$ or $\Der S$. More precisely, we will consider the algebraic transfer that uses the derived functors $-\lotimes_R S:\Der R\to \Der S$ and $\RHom_R(S,-):\Der R\to \Der S$  of the covariant, respectively contravariant, extension functors associated to $\lambda$, and the topological transfer (for cosilting objects of cofinite type) via the canonical map $\lambda^\star:\Spec S\to \Spec R$. In this way we continue the investigations realized in \cite{HST20} for the case of tilting modules and in \cite{Br20} for silting complexes of length $2$. For other contributions in the study of the transfer of some related properties via ring extensions we refer to the already mentioned paper \cite{RG} (projective modules), but also to \cite{CK16} (injective modules), \cite{Mi95} (compact tilting complexes) and \cite{To05} ($n$-tilting and $n$-cotilting modules). 

We will prove in Theorem \ref{ascend} that the silting property \textit{ascends} along $\lambda$  (i.e., the functor $-\lotimes_R S$ preserves silting objects), and that the same is true for pure-injective cosilting objects via the functor $\RHom_R(S,-)$. Moreover, these functors preserve the (co)finiteness type properties, see Theorem \ref{new-ascend-(co)finite}. For (co)silting objects of (co)finite type we can also apply a topological transfer since $\lambda^\star$ is continuous with respect to Hochster's topology, Theorem \ref{topological-transfer}. In fact, this transfer coincides up to equivalence to that realized by using the derived extension functors. Moreover, if $\lambda$ is faithfully flat then $\lambda^\star$ is also closed, and this allows us to identify all (co)silting complexes of (co)finite type from $\Der S$ that can be obtained, up to equivalence, by ascending a (co)silting object from $\Der R$. 

On the other side, as in the case of $n$-tilting modules (cf. \cite{HST20}), it is not clear if the descent of the (co)silting property is valid for all commutative rings. More precisely, the (co)silting property \textit{descends} with respect to a faithfully flat morphism $\lambda:R\to S$ of commutative rings if the functor $-\lotimes_R S$ reflects the silting objects (respectively, $\RHom_R(S,-)$ reflects the cosilting objects). In the last part of the paper we study the property for bounded (co)silting complexes. Even though we are not able to prove the descent property for all bounded complexes over commutative rings, we can indicate many situations when it is valid. We prove in Theorem \ref{a-cosilting} that the cosilting property descends for complexes that are duals of complexes of projectives. Note that this result is satisfactory for noetherian rings since in this case all bounded cosilting complexes have this form, up to equivalence. For the silting case, we prove in Corollary \ref{silt-zar-loc} that the bounded silting property is Zariski local. Moreover, we provide in Remark \ref{list-descend} an extensive list of rings for which the silting property descends for bounded complexes of projectives, cf. Proposition \ref{descent-ff-gen}. This list includes all noetherian rings, all rings that are of finite injective dimension or finite pure global dimension, in particular all rings of cardinality $\aleph_n$ for some integer $n \geq 0$. Also, the $n$-tilting property descends for these rings (Corollary \ref{tilt-descent-ff-gen}).

Unfortunately, we were not able to prove that the properties used in Lemma \ref{char-silt} to characterize (co)silting objects in $\Der R$ descend along faithfully flat morphisms (in \cite{Br20} and \cite{HST20} it is proved that these properties descend together for complexes of length $2$ and for $1$-tilting modules, but the proofs cannot be extended to general bounded complexes). In order to avoid this obstruction, we use two results that can be of independent interest. More precisely, in Theorem \ref{char-silt-bounded} it is proven that for bounded silting complexes the closure of the class $T^{\perp_{>0}}$ with respect to direct sums can be replaced by the weaker condition $\Add(T)\subseteq T^{\perp_{>0}}$. This generalizes a recent result proved for $n$-tilting modules by Positselski and \v S\v tov\'\i\v cek \cite[\S 2]{PoS}. The other one is presented in Theorem \ref{loc-pure-inj}, where it is proved that if $R$ is a commutative ring such that $\Der R$ coincides with its smallest localizing subcategory that contains all pure-injective objects then for every injective morphism of rings $\lambda:R\to S$ that is pure, the smallest localizing subcategory of $\Der R$ that contains $S$ is $\Der R$. 

Although in this paper we are mostly interested in commutative rings, for the descending property of silting objects we also need to consider a non-commutative setting: Marks and Vit\'oria had notice in \cite{MV18} that every bounded silting complex can be interpreted as a tilting module in the category of representations of a Dynkin quiver. Therefore, in Section \ref{sect-silt-cosilt} the results will be presented for general rings. In Section \ref{sect-ascend} we study the ascent properties for silting and cosilting for some morphisms that are induced by a ring morphism whose domain is commutative (note that for general rings the property to be silting does not ascend, e. g. \cite[Example 2.5]{Br20}). The next two sections are dedicated to extensions induced by morphisms of commutative rings. In Section \ref{Sect-cos-com} we study the transfer of (co)silting complexes of cofinite type by using both the algebraic and topological transfers mentioned above. In the last section we study the descent property for bounded silting complexes.

If $R$ is a ring then $\ModR$ will denote the category of all right $R$-modules, and $\Der R$ will be the unbounded derived category of $\ModR$. We recall that $\Der R$ is a triangulated category, and we will denote by ${\_}[1]$ the shift. The $i$-th homology functor is denoted by $H^i:\Der R\to Ab$.


\section{Preliminaries}

\subsection{Derived functors} Let $R$ and $S$ be two rings and let $X\in\Der {R\op\otimes_\Z S}$ be a complex of $(R,S)$-bimodules. Then the derived tensor product \[-\lotimes_RX:\Der R\to\Der S\] is the left adjoint of the derived hom functor \[\RHom_S(X,-):\Der S\to\Der R.\]
The link between the right derived hom functor and the usual hom functor is expressed by the natural isomorphism  
\[H^i\RHom_R(X,Y)=\Hom_{\Der R}(X,Y[i]).\]
Moreover the adjunction isomorphism 
\[\Hom_{\Der S}(-\lotimes_RX,-)\cong\Hom_{\Der R}(-,\RHom_S(X,-))\]
has an enriched version 
\[\RHom_S(-\lotimes_RX,-)\cong\RHom_R(-,\RHom_S(X,-)),\] the initial isomorphism being recovered by taking $H^0$.  
Letting $X$ to run over all objects in  $\Der {R\op\otimes_\Z S}$ we get two bifunctors. In particular, {if $R$ is commutative, }the derived tensor product induces a monoidal structure on the category $\Der R$, the unit being $R$ itself, and the internal hom being $\RHom_R(-,-)$.

Using these functors one can define two dualities on $\Der R$, namely \[(\ )^*:=\RHom_R(-,R):\Der R\to\Der{R\op}\] and \[(\ )^+:=\RHom_\Z(-,\Q/\Z):\Der R\to\Der{R\op}.\]
We denote by the same symbols the respective dualities going in the other sense, namely $\Der{R\op}\to\Der R$. 

We record some basic properties of these functors. Recall that an object $C\in \Der R$ is \textit{compact} if the covariant functor $\Hom_{\Der R}(C,-)$ commutes with respect to direct sums (this definition is valid for general triangulated categories). In the case of $\Der R$ the compact objects are those objects that are isomorphic to bounded complexes of finitely generated projective modules.  

\begin{lemma}\cite[Lemma 21]{AH21}. Let $X,C\in\Der R$ and $Y\in\Der{R\op}$, with $C$ compact. Then:
\begin{enumerate}
    \item For any $n\in\Z$ we have $H^n(X^+)\cong H^{-n}(X)^+$ and $H^n(X^+)=0$ \iff $H^{-n}(X)=0$.
    \item There are natural isomorphisms
    \[\RHom_{R\op}(Y,X^+)\cong(X\lotimes_RY)^+\cong\RHom_R(X,Y^+),\]
    \[\RHom_R(C,X)\cong X\lotimes_RC^*,\]
    and 
    \[\RHom_R(C,X)^+\cong C\lotimes_RX^+.\]
\end{enumerate}
\end{lemma}

\subsection{Local properties of objects of $\Der R$}\label{subsec-locality} Let $\mathfrak{P}$ be some property which can be satisfied by an object of the unbounded derived category of a commutative ring. Given a commutative ring $R$, let $\mathfrak{P}_R$ denote the full subcategory of objects of $\Der R$ satisfying $\mathfrak{P}$. 

In general, it is possible to extend the consideration of the property $\mathfrak{P}$ from the affine setting of modules over commutative rings to the non-affine setting of the category $\Qcoh X$ of quasi-coherent sheaves over a scheme $X$ in the following way. Given a scheme $X$, let $\Der X = \Der {\Qcoh X}$ denote the unbounded derived category of the category of quasi-coherent sheaves over $X$. Let $\mathcal{O}_X$ denote the structure sheaf on $X$. For any open subset $U \subseteq X$ we have the module of sections functor ${\_}(U): \Qcoh X \to \Mod {\mathcal{O}_X(U)}$. If $U$ is open affine then this functor is exact \cite[Lemma 30.2.2.]{stacks}, and so it extends naturally to a triangle functor ${\_}(U): \Der X \to  \Der{\mathcal{O}_X(U)}$. Then we say that an object $\mathcal{M} \in \Der X$ satisfies the property $\mathfrak{P}$ if the section object $\mathcal{M}(U) \in \Der{\mathcal{O}_X(U)}$ belongs to $\mathfrak{P}_{\mathcal{O}_X(U)}$ for {\em any} open affine subset $U$ of $X$.

Well-behaved properties of quasi-coherent sheaves are expected to be verifiable on any choice of open affine cover of $X$. A property $\mathfrak{P}$ is therefore called {\em (Zariski) local} if for any scheme $X$, any open affine cover $X = \bigcup_{i \in I}U_i$, and any $\mathcal{M} \in \Der X$, we have that $\mathcal{M}$ satisfies $\mathfrak{P}$ whenever $\mathcal{M}(U_i) \in \mathfrak{P}_{\mathcal{O}_X(U_i)}$ for all $i \in I$. Thanks to the following result, called the {\em Affine Communication Lemma}, there is a purely module-theoretic criterion sufficient to check that a property is local.

\begin{lemma}\label{affine-communication}
   Let $\mathfrak{P}$ be a property of objects of $\Der R$ where $R$ runs over commutative rings. Assume the following two conditions for any commutative ring $R$:
   \begin{enumerate}
       \item If $X \in \Der R$ satisfies $\mathfrak{P}$ then $X \lotimes_R R[f^{-1}] \in \mathfrak{P}_{R[f^{-1}]}$ for any element $f \in R$,
       \item Let $f_0,f_1,\ldots,f_{n-1} \in R$ be elements such that $R = \sum_{i<n}f_i R$ and consider the ring homomorphism $R \to S := \prod_{i<n}R[f_i^{-1}]$. Then for any $X \in \Der R$, if $X \lotimes_R S \in \mathfrak{P}_S$ then $X \in \mathfrak{P}_R$.
   \end{enumerate}
   Then the property $\mathfrak{P}$ is local.
\end{lemma}
\begin{proof}
   The proof is completely analogous to the non-derived version in \cite[Lemma 3.4]{EGT}, see also \cite[Lemma 2.1]{HST20}.
\end{proof}
It is often the case that a local property satisfies a stronger variant than the conditions of Lemma~\ref{affine-communication}. Following \cite[Definition 3.5]{EGT}, we say that a property $\mathfrak{P}$ is an {\em ad-property (=ascent-descent property)} if the following two conditions hold for any commutative ring $R$ and $X \in \Der R$:
\begin{enumerate}
       \item If $X \in \mathfrak{P}_R$ then $X \lotimes_R S \in \mathfrak{P}_{S}$ for any flat ring morphism $R \to S$,
       \item If $X \lotimes_R S \in \mathfrak{P}_S$ then $X \in \mathfrak{P}_R$ for any faithfully flat ring morphism $R \to S$.
\end{enumerate}

\subsection{Phantoms} 
We say that a morphism $\varphi:U\to V$ in $\Der R$ is a \textit{phantom} if for every compact object $C\in\Der R$ we have $\Hom_{\Der R}(C,\varphi)=0$, \cite{krause-00}. A triangle $X\to Y\to Z\overset{\delta}\to X[1]$ is called \textit{pure} if $\delta$ is a phantom. 

\begin{rem}
The notion of phantom is used in many triangulated categories. For instance, in the homotopy category of spectra a map $f$ is a phantom if and only if all maps induced by $f$ on homologies are zero. In the case of derived categories of modules this property is not true. In order to see this, we consider in  $\Der \Z$ the morphism $\varphi$ induced by the following morphism of complexes
$$\xymatrix{\cdots\ar[r]& 0\ar[r]\ar[d] & \Z\ar[r]^{2\cdot\_} \ar[d]^{1_\Z} & \Z \ar[r]\ar[d] & 0 \ar[r]\ar[d] & \cdots \\
\cdots\ar[r]& 0\ar[r] & \Z\ar[r] & 0 \ar[r] & 0 \ar[r] & \cdots}
.$$
It is easy to see that $H^n(\varphi)=0$ for all $n\in\Z$. But $\varphi\neq 0$ (since it is not equal to 0 in the homotopy category), so it is not a phantom since its domain is a compact object.
\end{rem}

\begin{lemma}\cite[Lemma 2.6]{AH21}\label{dual-phantom}
A morphism $\varphi:U\to V$ in $\Der R$ is a phantom if and only if $\varphi^+=0$. 
\end{lemma}

Assume that $\lambda:R\to S$ is a flat morphism of commutative rings and that $A$ is an $R$ algebra such that $R$ is central. We consider the $S$-algebra $B=A\otimes_R S$. Applying the functor $-\lotimes_R S=-\otimes_A B:\Der{A}\to \Der B$, we obtain a natural morphism (of $R$-modules) $\Hom_{\Der{A}}(C,Y)\to \Hom_{\Der{B}}(C\lotimes_R S,Y\lotimes_R S)$ that induces a natural morphism of $S$-modules $$\Xi:\Hom_{\Der{A}}(C,Y)\otimes_R S\to \Hom_{\Der{B}}(C\lotimes_R S,Y\lotimes_R S).$$ These morphisms are also described in \cite[Lemma 3]{Zim}.

The proof of the following lemma follows the lines of the proofs presented in \cite[Lemma 2.2]{NT}, \cite[Lemma 3]{Zim} for noetherian rings, where $C$ may be a complex of finitely generated modules. 

\begin{lemma} \label{lemma-zim}
Let $\lambda:R\to S$ be a flat morphism of commutative rings. Suppose that $A$ is an $R$-algebra such that $R$ is central and that $B=A\otimes_R S$. If $C,X\in\Der{A}$are two complexes such that $C$ is compact, the natural morphism $$\Xi:\Hom_{\Der{A}}(C,X)\otimes_R S\to \Hom_{\Der{B}}(C\lotimes_R S,X\lotimes_R S)$$ is an isomorphism.
\end{lemma}

\begin{proof} Observe that we can assume that $C$ is a bounded complex of finitely generated projective modules, so we can view the homomorphisms from $\Hom_{\Der{A}}(C,X)$ as morphisms in the homotopy category. As in the above mentioned proof we apply the induction on the length of $C$. If $C$ is of length $1$, it follows that $C$ is a finitely generated projective module concentrated in a degree $i$. It follows that $\Hom_{\Der{A}}(C,X)\cong \Hom_A(C^i,H^i(X))$ and, using the flatness of $S$, we obtain the isomorphisms \begin{align*}
    \Hom_{\Der{A}}(C\lotimes_R S,X\lotimes_R S)& \cong \Hom_B(C^i\otimes_R S,H^i(X\otimes_R S)) \\  &\cong \Hom_B(C^i\otimes_R S,H^i(X)\otimes_R S).\end{align*} We apply \cite[Lemma 3.2.4]{EJ} to conclude that, in this particular case, $\Xi$ is an isomorphism. 

Assume that the statement is valid for all bounded complexes of projectives of length at most $n-1$, and suppose that $C$ has length $n$ ($n\geq 2$). Using the brutal truncations we observe that we can embed $C$ in a triangle 
$$C'\to C\to C''\overset{+}\to$$ such that $C'$ and $C''$ are complexes of finitely generated projective modules of length at most $n-1$. 

We obtain the commutative diagram 
$$\xymatrix{\Hom_{\Der{A}}(C',X[-1])\otimes_R S\ar[r]^{\cong}\ar[d] & \Hom_{\Der{B}}(C'\lotimes_R S,X[-1]\lotimes_R S)\ar[d]\\ 
\Hom_{\Der{A}}(C'',X)\otimes_R S\ar[r]^{\cong}\ar[d] &  \Hom_{\Der{B}}(C''\lotimes_R S,X\lotimes_R S)\ar[d]\\ 
\Hom_{\Der{A}}(C,X)\otimes_R S\ar[r]\ar[d] & \Hom_{\Der{B}}(C\lotimes_R S,X\lotimes_R S)\ar[d]\\
\Hom_{\Der{A}}(C',X)\otimes_R S\ar[r]^{\cong}\ar[d] & \Hom_{\Der{B}}(C'\lotimes_R S,X\lotimes_R S)\ar[d] \\ 
\Hom_{\Der{A}}(C'',X[1])\otimes_R S\ar[r]^{\cong}& \Hom_{\Der{B}}(C''\lotimes_R S,X[1]\lotimes_R S),
}$$ 
where the horizontal arrows are the corresponding maps $\Xi$. The conclusion follows from five lemma. 
\end{proof}

\begin{rem}
If $\lambda:R\to S$ is faithfully flat then the derived functor $-\lotimes_R S:\Der R\to \Der S$ does not reflect zero-morphisms. In order to see this, let consider a prime $p$ and the faithfully flat inclusion $\Z_{(p)}\to \mathbb{J}_p,$ where $\Z_{(p)}$ it the (local) ring of all rational numbers whose denominator is coprime with $p$ and $\mathbb{J}_p$ is the ring of $p$-adic integers. Observe that 
$$\Hom_{\Der{\Z_{(p)}}}(\mathbb{Q},\Z_{(p)}[1])\cong\Ext_{\Z_{(p)}}(\mathbb{Q},\Z_{(p)})\neq 0.$$ However, $$\Hom_{\Der{\mathbb{J}_{p}}}(\mathbb{Q}\lotimes_{\Z_{(p)}}\mathbb{J}_{p},\Z_{(p)}\lotimes_{\Z_{(p)}}\mathbb{J}_{p}[1])\cong\Ext_{\mathbb{J}_{p}}(\mathbb{Q}\lotimes_{\Z_{(p)}}\mathbb{J}_{p},\mathbb{J}_{p})= 0,$$ since $\mathbb{J}_{p}$ is pure-injective and $\mathbb{Q}\lotimes_{\Z_{(p)}}\mathbb{J}_{p}$ is flat.  
However, $-\lotimes_R S$ reflects the phantoms.
\end{rem}

\begin{prop}\label{phantoms}
Let $\lambda:R\to S$ be a faithfully flat morphism of commutative rings. Suppose that $A$ is an $R$-algebra such that $R$ is central and that $B=A\otimes_R S$. If $\phi:Z\to X[1]$ is a map $\Der{A}$ such that the induced map  $Z\lotimes_RS\to X[1]\lotimes_RS$ is phantom in $\Der B$ then $\phi$ is phantom too. 
\end{prop}

\begin{proof}
For every compact $C$ in $\Der A$ the object $C\lotimes_RS$ is compact in $\Der B$, hence $\Hom_{\Der B}(C\lotimes_RS,\phi\lotimes_R S)=0$. Using Lemma \ref{lemma-zim} it follows that $\Hom_{\Der A}(C,\phi)\otimes_R S=0$, hence $\Hom_{\Der A}(C,\phi)=0$. 
%
\end{proof}

\section{Silting and cosilting in $\Der R$}\label{sect-silt-cosilt}

\subsection{TTF-triples in triangulated categories}
We recall some standard notions and notations in triangulated categories. Let $\D$ be a triangulated category. 
For a subcategory $\C\subseteq\D$ and a subset $I\subseteq\Z$ of the set of integers, we denote 
\[\C^{\perp_I}=\{X\in\D\mid\Hom_\D(C,X[i])=0,\hbox{ for all }C\in\C\hbox{ and all }i\in I\}\]
and 
\[{^{\perp_I}\C}=\{X\in\D\mid\Hom_\D(X,C[i])=0,\hbox{ for all }C\in\C\hbox{ and all }i\in I\}.\] Often the set $I$ is symbolized as $\leq n$, $<n$, $=n$ or similar, where $n\in\Z$, with the obvious meaning. 

A {\em torsion pair} in $\D$ is a pair $(\U,\V)$ of subcategories $\D$ such that:
\begin{enumerate}
    \item Both $\U$ and $\V$ are closed under direct summands.
    \item $\U^{\perp_0}=\V$ and $\U={^{\perp_0}\V}$.
    \item Every objects $X\in\D$ lies in a triangle $U\to X\to V\to U[1]$ with 
    $U\in\U$ and $V\in\V$. 
\end{enumerate}
It is clear that if $(\U,\V)$ is a torsion pair, then $\U$ is closed under coproducts and $\V$ is closed under products. 
A {\em TTF-triple} in $\D$ is a triple that  $(\U,\V,\W)$ such that both $(\U,\V)$ and $(\V,\W)$ are torsion pairs. Let $\CS$ be a set of objects from $\D$. The TTF triple is \textsl{generated by $\CS$} if $\V=\CS^{\perp_0}$. An object $C\in \D$ is called \textit{compact} if the covariant functor $\Hom_\D(C,-)$ commutes with respect to direct sums. If $\CS$ is a set of compact objects then we will say that the TTF-triple is \textit{compactly generated}. 

A torsion pair is called \textit{t-structure} (\textit{co-t-structure}) if in addition $\U$ is closed under positive (respectively negative shifts). Note that the closure of $\U$ under positive (negative) shifts, that is $\U[1]\subseteq\U$ (respectively $\U[-1]\subseteq\U$) is equivalent to $\V[-1]\subseteq\V$ (respectively $\V[1]\subseteq\V$). In these cases, we say that $\U$ is the aisle, respectively $\V$ is the coaisle of the (co)-t-structure $(\U,\V)$. 

\subsection{Silting and cosilting objects in derived categories}
Let $T\in\Der R$. We say that $T$ is a {\em silting object} if the pair $(T^{\perp_{>0}},T^{\perp_{\leq 0}})$ is a t-structure in $\Der R$. In this case, the induced t-structure  is called the {\em silting t-structure} induced by $T$. A class of objects is a \textit{silting class} if it is an aisle of a silting t-structure. 
Dually, an object $C\in\Der{R}$ is called \textit{cosilting} if $({^{\perp_{\leq0}}C},{^{\perp_{>0}}C})$ is a t-structure, called the {\em cosilting t-structure} induced by $C$. A coaisle of a cosilting t-structure is called a \textit{cosilting class}. Two (co)silting objects are \textit{equivalent} if they induce the same (co)silting class.

\begin{rem}\label{functors-classes}
We know by \cite[Lemma 4.5]{PS18} that two silting objects $T_1$ and $T_2$ are equivalent \iff $\Add(T_1)=\Add(T_2)$. Similarly, two cosilting objects $C_1$ and $C_2$ are equivalent \iff $\Prod(C_1)=\Prod(C_2)$. 
\end{rem}

\begin{lemma}\label{char-silt}
 An object $T\in\Der R$ is silting \iff
 \begin{enumerate}[{\rm (S1)}]
        \item $T\in T^{\perp_{>0}}$. 
        \item $T^{\perp_{>0}}$ is closed under coproducts.
        \item $T$ generates $\Der R$, that is $T^{\perp_{\Z}}=\{0\}$. 
    \end{enumerate}
 Dually, a pure-injective object $C\in\Der{R}$ is cosilting \iff
    \begin{enumerate}[{\rm (C1)}]
        \item $C\in{^{\perp_{>0}}C}$. 
        \item ${^{\perp_{>0}}C}$ is closed under products.
        \item $C$ cogenerates $\Der{R}$, that is ${^{\perp_{\Z}}C}=\{0\}$. 
    \end{enumerate}
\end{lemma}

\begin{proof}
 The characterization of the silting complexes is a particular case of \cite[Proposition 4.13]{PS18}.
 For the cosilting case, the argument is dual, the only modification occurs when we need to show that
 $C$ induces a t-structure whose aisle is ${^{\perp_{\leq0}}C}$ (and consequently the coaisle is the smallest cosuspended and closed under products subcategory of $\Der R$). Here we use the additional hypothesis that $C$ is pure-injective and then the claim follows by \cite[Corollary 5.11]{LV20} (the main result of \cite{N18} is essentially the reason why no extra assumption is needed on the silting side.)
\end{proof}

\begin{rem}\label{TTF-notations} (1) A $t$-structure  $(\Y,\CZ)$ is silting  \iff it extends to a TTF-triple $(\X,\Y,\CZ)$ which is non-degenerate, suspended and generated by a set of objects in $\Der R$, \cite[Theorem 4.11]{An-19}. If $(\X,\Y,\CZ)$ is induced by the silting object $T$, we will denote it by $(\X_T,\Y_T,\CZ_T)$

(2) A silting object $T$ is bounded, that is, isomorphic in $\Der R$ to a bounded complex of projectives, 
if and only if  $(\X_T,\Y_T,\CZ_T)$ is intermediate and suspended, see \cite[Theorem 2.1]{AH21}. In particular, this TTF-triple is compactly generated. 

(3) Dually, a t-structure $(\U,\V)$ in $\Der{R\op}$ 
is induced by a pure--injective cosilting $C$ object \iff it extends to a non-degenerate cosuspended TTF triple $(\U,\V,\W)$ which is homotopically smashing, \cite[Theorem 4.6]{La}. In this case we will write $(\U,\V,\W)=(\U_C,\V_C,\W_C)$. 

(4) A cosilting object $C$ is bounded, that is, isomorphic in $\Der R$ to a bounded complex of injectives if and only if the associated TTF-triple $(\U_C,\V_C,\W_C)$ is cointermediate and cosuspended, see \cite[Theorem 2.12]{AH21}. As before, in this case $(\U_C,\V_C,\W_C)$ is homotopically smashing.
\end{rem}

\subsection{(Co)Silting of (co)finite type}

For a not necessarily commutative ring $R$, there is a 1-to-1 correspondence (see \cite[Theorem 3.1]{AH21}) 
\[\left\{\begin{array}{c}\hbox{compactly generated TTF }\\ \hbox{ triples in }\Der{R}\end{array}\right\}\longrightarrow\left\{\begin{array}{c}\hbox{compactly generated TTF }\\ \hbox{ triples  in }\Der{R\op}\end{array}\right\} \]
which assigns the TTF triple in $\Der R$ generated by a set $\C$ of compacts in $\Der R$ to the TTF triple in $\Der{R\op}$ generated by the set of compacts $\C^*=\{C^*\mid C\in\C\}$. 

A silting object $T\in\Der R$ is of \textit{finite type} if the TTF triple $(\X_T,\Y_T,\CZ_T)$ is compactly generated. A cosilting object $C$ is of \textit{cofinite type} if the TTF triple $(\U_C,\V_C,\W_C)$ is compactly generated. 

\begin{rem}
From  \cite{SS20} it follows that all compactly generated t-structures in $\Der R$ are homotopically smashing, hence all cosilting objects of cofinite type from $\Der R$ are pure-injective. {The converse is not true in general, but it holds if $R$ is commutative noetherian \cite[Corollary 2.14]{HN21}.}
\end{rem}

\begin{thm}\label{silt-cosilt-dual}\cite[Theorem 3.3]{AH21} The assignment $T\mapsto T^+$ induces: 
\begin{enumerate}[{\rm (1)}]
\item an injective map from the set of equivalence classes of silting objects of finite type from $\Der R$ to the set of equivalence classes of cosilting objects of cofinite type from $\Der{R\op}$;

\item a bijective map between the set of equivalence classes of bounded silting complexes from $\Der R$ and the set of equivalence classes of bounded cosilting complexes of cofinite type from $\Der{R\op}$.
\end{enumerate}
\end{thm}

\begin{rem}\label{dual-definable}
From the proof of \cite[Lemma 3.2]{AH21} it follows that in the above correspondences, the classes $\Y_{T}$ and $\V_{T^+}$ are dual definable classes, i.e. $Y\in \Y_T$ if and only if $Y^+\in \V_{T^+}$ and $V\in V_{T^+}$ if and only if $V^+\in \Y_T$.
\end{rem}

\begin{lemma}\label{t-str-H}
 For an object $T\in\Der R$ and a subset $I\subseteq\Z$ we have: 
 \begin{enumerate}[{\rm (a)}]
    \item $T^{\perp_I}=\{X\in\Der R\mid H^n\RHom_R(T,X)=0,\hbox{ for all }n\in I\}$. 
    \item ${^{\perp_I}(T^+)}=\{X\in\Der{R\op}\mid H^n(T\lotimes_RX)=0,\hbox{ for all }n\in-I\}$, \newline where $-I=\{-i\mid i\in I\}$.
\end{enumerate}
\end{lemma}

\begin{proof}
The equality (a) follows by the isomorphism \[\Hom_{\Der R}(T,X[n])\cong H^n\RHom_R(T,X),\] valid for any $n\in\Z$.  

For the equality (b), note  that for all $n\in\Z$ we have the isomorphisms:
\begin{align*}
\Hom_{\Der{R\op}}(X,T^+[n])&\cong H^n\RHom_{R\op}(X,T^+)\\
&=H^n\RHom_{R\op}(X,\RHom_\Z(T,\Q/\Z))\\ &\cong H^n\RHom_\Z(T\lotimes_RX,\Q/Z)\\
&=H^n(T\lotimes_RX)^+
\end{align*}
and $H^n(T\lotimes_RX)^+=0$ \iff $H^{-n}(T\lotimes_RX)=0$.
\end{proof}

\begin{cor}\label{t-str-by-H}
Let $T\in\Der R$. 
If $T$ is a silting object and $(\X_T,\Y_T,\CZ_T)$ is the TTF triple associated to $T$, then 
\begin{enumerate}
    \item[{\rm (a)}] $\Y_T=\{X\in\Der R\mid H^n\RHom_R(T,X)=0,\hbox{ for all }n>0\}$. 
    \item[{\rm (b)}] $\CZ_T=\{X\in\Der R\mid H^n\RHom_R(T,X)=0,\hbox{ for all }n\leq 0\}$.
\end{enumerate}
If $T^+$ is cosilting (in $\Der{R^{\mathrm{op}}}$) and $(\U_{T^+},\V_{T^+},\W_{T^+})$ is the corresponding TTF triple, then:
\begin{enumerate}
    \item[{\rm (c)}] $\U_{T^+}=\{X\in\Der{R^{\mathrm{op}}}\mid H^n(T\lotimes_RX)=0,\hbox{ for all }n\geq0\}$.
    \item[{\rm (d)}] $\V_{T^+}=\{X\in\Der{R^{\mathrm{op}}}\mid H^n(T\lotimes_RX)=0,\hbox{ for all }n<0\}$.
\end{enumerate}
\end{cor}

 \begin{cor}\label{tensor-gen}
Let $T\in\Der R$ be an object such that $T^+$ is cosilting in $\Der R$. Then, for every $X\in\Der{R^{\mathrm{op}}}$ we have $T\lotimes_RX=0$ \iff $X=0$.
 \end{cor}

\begin{proof} The converse implication is obvious, so we need to prove only the direct one. Let $X\in\Der{R^{\mathrm{op}}}$ with the property $T\lotimes_RX=0$. Then $H^n(T\lotimes_RX)=0$ for all $n\in\Z$ and Lemma \ref{t-str-by-H} implies that $X\in\U_{T^+}\cap\V_{T^+}=\{0\}$. 
\end{proof}

\subsection{Bounded (co)silting complexes}\label{section-bounded}
In what follows, we will be particularly interested in silting complexes which are \textit{bounded} in the sense that they are isomorphic in the derived category to bounded complexes of projective modules. This is a natural condition to consider since it is part of the definition of $n$-tilting modules, and also it ensures a well-behaved theory of derived equivalences \cite{PS18}. It turns out that the characterization of Lemma~\ref{char-silt} simplifies considerably in this situation. 

For a (full) subcategory $\C$ of $\Der R$, {let $\add(C)$ (resp., $\Add(C)$) denote the subcategory formed by all direct summands of finite coproducts (resp., all coproducts) of objects from $\C$. Similarly, $\Prod(\C)$ is formed by all direct summands of arbitrary products of objects of $\C$. If $\C = \{X\}$ for some object $X$, we drop the curly brackets in the notation.}
We denote by $\susp(\C)$ (resp. $\cosusp(\C)$) the suspended (resp. cosupended) closure of $\C$ in $\Der R$, that is, the smallest full subcategory of $\Der R$ containing $\C$ and closed under direct summands, extensions and $[1]$ (resp., $[-1]$). Both $\susp(\C)$ and $\cosusp(\C)$ admit an explicit description as follows. Let $\CE^+_0 = \add \{C[n] \mid C \in \C, n \geq 0\}$ and $\CE^-_0 = \add \{C[n] \mid C \in \C, n \leq 0\}$. For $i>0$, define inductively subcategory $\CE^+_i$ consisting of all $X \in \Der R$ fitting into a triangle $E_0 \to X \to E_{i-1} \to E_0[1]$ with $E_0 \in \CE^+_0$ and $E_{i-1} \in \CE^+_{i-1}$; subcategories $\CE^-_i$ are defined in the analogous way. Then one can easily check that $\susp(\C) = \bigcup_{i \geq 0}\CE^+_i$ and $\cosusp(\C) = \bigcup_{i \geq 0}\CE^-_i$. This description has a consequence important in what follows:

\begin{lemma}\label{susp-coprod}
Suppose that $\C$ is a full subcategory of $\Der R$ closed under products (resp. coproducts). If $X \in \susp(\C)$ then $X^I \in \susp(\C)$ (resp. $X^{(I)} \in \susp(\C)$) for any set $I$, and the same holds for $\cosusp(\C)$.
\end{lemma}
\begin{proof}
Assume $\C$ is closed under products and let $I$ be a set. Since $X \in \susp(\C)$, there is by the above description some $i \geq 0$ such that $X \in \CE^+_i$. We prove $X^I \in \CE^+_i$ by induction on $i \geq 0$. If $i=0$ then $X \in \add\{C[n] \mid C \in \C, n \geq 0\}$. Since $\C$ is closed under products, clearly $X^I \in \add\{C[n] \mid C \in \C, n \geq 0\}$. If $i>0$ then there is a triangle $E_0 \to X \to E_{i-1} \to E_0[1]$ with $E_0 \in \CE^+_0$ and $E_{i-1} \in \CE^+_{i-1}$. Then the triangle $E_0^I \to X^I \to E_{i-1}^I \to E_0^I[1]$ together with the induction hypothesis shows that $X^I \in \CE^+_{i}$. The case of coproduct closure and/or the cosuspended closure is handled the same way.
\end{proof}
The following result extends \cite[Theorem 2.3 and Theorem 4.3]{PoS} from the case of $n$-(co)tilting modules, and the proof idea is similar to \emph{loc. cit.} with some necessary modifications. For basics on homotopy limits and colimits in triangulated categories we refer to \cite[\S 1.6]{Neeman}. We will also use the notations 
$\Der{R}^{\geq n}=\{X\in \Der R\mid H^{i}(X)=0 \textrm{ for all }i<n\}$  and $\Der{R}^{\leq n}=\{X\in \Der R\mid H^{i}(X)=0 \textrm{ for all }i>n\}. $

\begin{thm}\label{char-silt-bounded}
Let $T\in\Der R$ be an object which is isomorphic to a bounded complex of projective $R$-modules. Then $T$ is silting \iff
 \begin{enumerate}[{\rm (Sb1)}]
        \item $\Add(T)\subseteq T^{\perp_{>0}}$. 
        \item $T$ generates $\Der R$, that is $T^{\perp_{\Z}}=\{0\}$. 
    \end{enumerate}
 Dually, let $C\in\Der R$ be an object which is isomorphic to a bounded complex of injective $R$-modules. Then $C$ is cosilting \iff
    \begin{enumerate}[{\rm (Cb1)}]
        \item $\Prod(C)\subseteq{^{\perp_{>0}}C}$. 
        \item $C$ cogenerates $\Der{R}$, that is ${^{\perp_{\Z}}C}=\{0\}$. 
    \end{enumerate}
\end{thm}
\begin{proof}
    We prove only the silting result, the cosilting one follows by a completely analogous argument. Put $(\Y,\CZ) = (T^{\perp_{>0}},T^{\perp_{\leq 0}})$ and our goal is to show that this constitutes a t-structure in $\Der R$. Without loss of generality, we may assume that $T$ is a complex of projective $R$-modules concentrated in degrees $-n+1,-n+2,\ldots,-1,0$ for some $n>0$.
    
    First, we show that for any $X \in \Der R$ there is a triangle $Y \to X \to Z \to Y[1]$ with $Y \in \Y$ and $Z \in \CZ$. Put $X_0:=X$ and define inductively a sequence of triangles 
    $$T^{(H_i)}[i] \xrightarrow{h_i} X_i \xrightarrow{f_i} X_{i+1} \to T^{(H_i)}[i+1]$$
    where $H_i = \Hom_{\Der R}(T[i],X_i)$ and $h_i$ is the canonical evaluation morphism. Since $\Hom_{\Der R}(T[i],h_i)$ is surjective for each $i \geq 0$ and by (Sb1), we see that 
    $$\Hom_{\Der R}(T[j],X_i) = 0 ~\textrm{ for all } j=0,1,\ldots,i-1.$$
    This construction defines an inductive system 
    $$X_0 \xrightarrow{f_0} X_{1} \xrightarrow{f_1} X_{2} \xrightarrow{f_2} X_{3} \xrightarrow{f_3} \cdots$$ 
    and we consider its homotopy colimit $Z = \hocolim_{i\geq 0}(X_i)$. Recall that for any $j \geq 0$ we have $Z = \hocolim_{i\geq j}(X_i)$, see \cite[Lemma 1.7.1]{Neeman}. Considering $X_j$ as the trivial homotopy colimit $X_j = \hocolim_{i\geq j}(X_j)$ of an inductive tower consisting of identities on $X_j$, \cite[Lemma 1.6.6]{Neeman}, we obtain a commutative square
    $$
    \begin{CD}
    \coprod_{i \geq j}X_i @>(1-\text{shift}(f_i))>> \coprod_{i \geq j}X_i \\
    @A\coprod_{i \geq j}f_{ji}AA @A\coprod_{i \geq j}f_{ji}AA\\
    \coprod_{i \geq j}X_j @>(1-\text{shift}(\mathrm{Id}_{X_j}))>> \coprod_{i \geq j}X_i \\
    \end{CD}
    $$
    where $f_{ji} = f_{i-1}\circ\cdots\circ f_{j+1}\circ f_j$ for $j<i$ and $f_{jj} =\mathrm{Id}_{X_j}$. For each $0 \leq j \leq i$, let $S_{ji}$ be the cone of $f_{ji}$. Then $S_{ji}$ is an iterated extension of the objects $T[k+1]^{(H_k)}$ where $k=j,j+1,\ldots,i-1$. In particular, since $T \in \Der{R}^{\leq 0}$, we have $S_{ji} \in \Der{R}^{\leq -j-1}$. By \cite[1.1.11]{BBD}, the square above extends to a diagram
    $$
    \begin{CD}
    \coprod_{i \geq j}S_{ji} @>>> \coprod_{i \geq j}S_{ji} @>>> C \\
    @AAA @AAA @AAA \\
    \coprod_{i \geq j}X_i @>(1-\text{shift}(f_i))>> \coprod_{i \geq j}X_i @>>> Z\\
    @A\coprod_{i \geq j}f_{ji}AA @A\coprod_{i \geq j}f_{ji}AA @AAA\\
    \coprod_{i \geq j}X_j @>(1-\text{shift}(\mathrm{Id}_{X_j}))>> \coprod_{i \geq j}X_i @>>> X_j\\
    \end{CD}
    $$
    in which all squares commute and all rows and columns extend to triangles. Since $S_{ji} \in \Der{R}^{\leq -j-1}$, we have $\coprod_{i \geq j}S_{ji} \in \Der{R}^{\leq -j-1}$, and therefore $C \in \Der{R}^{\leq -j-1}$. Using that $\Hom_{\Der R}(T,\Der{R}^{\leq -n})=0$, we argue from the rightmost vertical triangle of the diagram that $\Hom_{\Der R}(T[i],Z) \cong \Hom_{\Der R}(T[i],X_{i+n+1}) = 0$ for any $i \geq 0$. This shows that $Z \in \CZ = T^{\perp_{\leq 0}}$.

    Now extend the map $X = X_0 \to Z$ to a triangle $Y \to X \to Z \to Y[1]$ and let us show that $Y \in \Y = T^{\perp_{>0}}$. The diagram above puts $Y$ into a triangle $$\textstyle Y \to \coprod_{i \geq 0}S_{0i} \to \coprod_{i \geq 0}S_{0i} \to Y[1].$$ Put $C_{0i} = S_{0i}[-1]$  for all $i > 0$, so that $Y$ is the mapping cone of $\coprod_{i \geq 0}C_{0i} \to \coprod_{i \geq 0}C_{0i}$. 
    
    Recall that $C_{0i}$ is an iterated extension of objects $T^{(H_0)},\ldots,T^{(H_{i-1})}[i-1]$, and therefore $C_{0i} \in \mathrm{susp}(\Add(T)) \subseteq \Y = T^{\perp_{>0}}$, the latter inclusion follows from (Sb1) and the obvious closure properties of $\Y$. It remains to show that also the coproduct $\coprod_{i \geq 0}C_{0i}$ belongs to $\Y$. Note that by the assumption on $T$, we clearly have $\Der{R}^{\leq -n} \subseteq \Y$. Put $m = n+1$. For any $i > m$, there is a triangle $C_{0m} \to C_{0i} \to C_{mi} \to C_{0m}[1]$, and we know that $C_{mi} \in \Der{R}^{\leq -m}$. Consider the coproduct triangle $$\textstyle \coprod_{i > m}C_{0m} \to \coprod_{i > m}C_{0i} \to \coprod_{i > m}C_{mi} \to \coprod_{i > m}C_{0m}[1].$$ Then $\coprod_{i > m}C_{mi} \in \Der{R}^{\leq -m} = \Der{R}^{\leq -n-1}$. Since $\Der{R}^{\leq -n} \subseteq \Y$, we see that $\coprod_{i > m}C_{0i} \in \Y$ if and only if $\coprod_{i > m}C_{0m} \cong C_{0m}^{(\omega)} \in \Y$. But the latter inclusion follows from Lemma~\ref{susp-coprod} because $C_{0m} \in \susp(\Add(T)) \subseteq \Y$. Therefore, $\coprod_{i > m}C_{0i} \in \Y$. Since we already know that $\coprod_{i \leq m}C_{0i} \in \Y$, we showed that $Y \in \Y$.

    Finally we need to show that $\Hom_{\Der R}(\Y,\CZ) = 0$. Take $X \in \Y$, and consider the same triangle $Y \to X \to Z \to Y[1]$ as constructed above. Since $Y, X \in \Y$, also $Z \in \Y$. But then $Z \in \Y \cap \CZ = T^{\perp_{\Z}}$. It follows by (Sb2) that $Z = 0$ and so $X \cong Y$. By the construction, $Y$ belongs to the smallest suspended and coproduct closed subcategory of $\Der R$ generated by $T$ (since we constructed $Y$ from $T$ by using extensions, $[1]$, and coproducts). Since $\Hom_{\Der R}(T,\CZ) = 0$, this shows $\Hom_{\Der R}(\Y,\CZ) = 0$, which in turn renders $(\Y,\CZ)$ a t-structure.
\end{proof}

We record here for further use a consequence of the above result that was already observed in \cite{PoS}.

\begin{cor}\cite[Corollary 3.6]{PoS} \label{Corollary 3.6-PoS}
Suppose that $T\in \Mod R$ is a module. Then $T$ is $n$-tilting (for a positive integer $n$) if and only if  
\begin{enumerate}[{\rm (i)}]
        \item $T$ is of projective dimension at most $n$;
				\item $\Add(T)\subseteq T^{\perp_{>0}}$;
        \item $T$ generates $\Der R$, that is $T^{\perp_{\Z}}=\{0\}$. 
\end{enumerate}
\end{cor}

\begin{proof}
We apply Theorem \ref{char-silt-bounded} together with \cite[Corollary 3.7]{Wei-isr}.   
\end{proof}

\section{Transfer of (co)silting objects using ring morphisms} \label{sect-ascend}

In this section we will use the following

\begin{setting}\label{setting-1}  Let $\lambda:R\to S$ be a morphism of commutative rings. Suppose that $A$ is an $R$-algebra such that $R$ is central and that $B=A\otimes_R S$. Then we have a ring morphism $A\to B$. The covariant, respectively contravariant, extension functor and the restriction functor induced by this morphism (we use the terminology from \cite{Ca}) and their derivates (see \cite{Ye}) can be described, up to natural isomorphism, in the following ways: 
\begin{itemize}
\item $-\otimes_A B\cong -\otimes_R S:\Mod A\to \Mod B$ (the covariant extension functor);
\item $-\lotimes_A B:\Der A\to \Der B$ (the derived covariant extension functor);
{if we assume that $A$ or $S$ is flat over $R$ then we have a natural isomorphism 
\begin{align*}
  (-\lotimes_AB)&\cong(-\lotimes_A(A\otimes_RS))\cong(-\lotimes_A(A\lotimes_RS))\\ &\cong((-\lotimes_AA)\lotimes_RS)\cong(-\lotimes_RS).  
\end{align*}}
\item $\Hom_A(B,-)\cong \Hom_R(S,-):\Mod A\to \Mod B$ (the contravariant extension functor);
\item $\RHom_A(B,-):\Der A\to \Der B$ (the derived  contravariant extension functor); {again, the supplementary assumption that $A$ or $S$ is flat over $R$ allows us to compute 
\begin{align*}
  \RHom_A(B,-)&\cong\RHom_A(S\otimes_RA,-)\cong\RHom_A(S\lotimes_R A,-)\\ &\cong\RHom_R(S,\RHom_A(A,-))\cong\RHom_R(S,-).  
\end{align*}}
\item $\Hom_B(B,-)\cong -\otimes_B B\cong \Hom_S(S,-)\cong -\otimes_S S:\Mod B\to \Mod A$ (the restriction functor);
\item $\RHom_B(B,-)\cong -\lotimes_B B\cong \RHom_S(S,-)\cong -\lotimes_S S:\Der B\to \Der A$ (the derived restriction functor).
\end{itemize}  

Note that the restriction functors do not change the module/complex. Therefore, there is no danger of confusion if we write $X$ instead of $\Hom_B(B,X)$ or $\RHom_B(B,X)$, respectively. 
Moreover, if $\A$ is a subcategory in $\Der A$, we denote by $\A\bigcap \Der B$ the subcategory of $\Der B$ defined by $$\A\bigcap \Der B=\{X\in \Der B\mid \RHom_B(B,X)=X\lotimes_B B\in\A\}.$$
\end{setting}

\subsection{The extension of scalars functors applied on (co)silting objects} We will prove that the derived covariant (contravariant) extension functors defined above preserve the silting (respectively, pure-injective cosilting) objects from $\Der R$. 

We start by observing that \cite[Lemma 1.10]{AJS} and \cite[Proposition 2.3]{H20} are still valid in our setting (with \textit{verbatim} proofs). 

\begin{lemma}\label{r-comm}
Let $R$ be commutative and $A$ be an $R$-algebra such that $R$ is central. Let $(\U,\V)$ be a t-structure in $\Der A$. For every $X\in\Dr^{\leq0}(R)$ we have:
\begin{enumerate}[{\rm (1)}]
    \item If $U\in\U$ then $U\lotimes_RX\in\U$.
    \item If $V\in\V$ then $\RHom_R(X,V)\in\V$.
\end{enumerate}
\end{lemma}

In the next theorem we will use the notation presented in Remark \ref{TTF-notations}.

\begin{thm}\label{ascend} Assume that we are in the hypotheses of Setting \ref{setting-1}, {and $A$ or $S$ is flat over $R$}. 
\begin{enumerate}[{\rm (I)}] \item Suppose that $T\in\Der A$ is silting. 
\begin{enumerate}[{\rm (1)}]
    \item The complex $T\lotimes_RS$ is silting in $\Der B$. 
    \item $\Y_{T\lotimes_R S}= \Y_T \bigcap \Der B$, and $\CZ_{T\lotimes_R S}= \CZ_T \bigcap \Der B$.
    \item $\X_T\lotimes_R S\subseteq \X_{T\lotimes_R S}$, and $\Y_T\lotimes_R S\subseteq \Y_{T\lotimes_R S}$.
    \item If $T$ is bounded then $T\lotimes_RS$ has the same property.  
    \end{enumerate} 
\item Suppose that $C\in\Der A$ is pure-injective cosilting.
\begin{enumerate}[{\rm (1)}]
    \item The complex $\RHom_R(S,C)$ is pure-injective cosilting in $\Der B$. 
    \item $\U_{\RHom_R(S,C)}= \U_C \bigcap \Der B$, and $\V_{\RHom_R(S,C)}= \V_C \bigcap \Der B$.
    \item $\RHom_R(S,\V_C)\subseteq \V_{\RHom_R(S,C)}$, and $\RHom_R(S,\W_C)\subseteq \W_{\RHom_R(S,C)}$.
    \item If $C$ is bounded cosilting, then $\RHom_R(S,C)$ has the same property.
 \end{enumerate}
\end{enumerate}
\end{thm}

\begin{proof} 
(I) (1) The ring homomorphism $\lambda$ induces a structure of $R$-module on $S$, and as a complex of $R$ modules it belongs to $\Dr^{\leq0}(R)$. Since $T$ is silting, it induces a t-structure $(T^{\perp_{>0}}, T^{\perp_{\leq 0}})$ in $\Der A$. From $T\in{T^{\perp_{>0}}}$ it follows by Lemma \ref{r-comm}  that $T\lotimes_R S\in T^{\perp_{>0}}$. 

For every $n\in\Z$ and every $X\in\Der B$, we have the adjunction isomorphism 
\[\Hom_{\Der A}(T,X[n])\cong\Hom_{\Der B}((T\lotimes_RS),X[n]).\]
Using this isomorphism for $n>0$ and $X=T\lotimes_RS$, together with  $T\lotimes_RS\in T^{\perp_{>0}}$, we deduce that the $B$-complex $T\lotimes_RS$ belongs to $(T\lotimes_RS)^{\perp_{>0}}$.

Further, the same isomorphism, applied for $n=0$ and $X=\coprod_{i\in I}X_i$, together with the fact that  (the derived functor of) the restriction of the scalars functor $\Der B\to \Der A$ preserves all coproducts, shows that the class $(T\lotimes_RS)^{\perp_{>0}}\subseteq \Der B$ is closed under coproducts. 

Finally, if $X\in\Der B$ and $X\in(T\lotimes_RS)^{\perp_{\Z}}$, it follows that $\Hom_{\Der A}(T,X[n])=0$ for all $n\in\Z$. Then $X$ is acyclic as an $A$-complex, hence it is acyclic as a $B$-complex.   

(2) These equalities are consequences of the isomorphism
$$\Hom_{\Der B}(T\lotimes_R S,Y)\cong \Hom_{\Der B}(T\lotimes_A A \lotimes_R S,Y)\cong \Hom_{\Der R}(T,\RHom_B(B,Y)).$$

(3) Let $Y\in \Y_{T}$. Then $\RHom_B(B,Y\lotimes_R S)=Y\lotimes_R S\in \CY_T$. For all $i>0$ we have
\begin{align*} & \Hom_{\Der B}(T\lotimes_R S, Y\lotimes_R S[i]) \cong \Hom_{\Der B}(T\lotimes_A B, Y\lotimes_R S[i]) \\  
 \cong &  \Hom_{\Der A}(T, \RHom_B(B,Y\lotimes_R S)[i])\cong  \Hom_{\Der A}(T, Y\lotimes_R S[i])=0.\end{align*} 

Let $X\in \X_T$ and $Y\in \Y_{T\lotimes_R S}$. Then $\RHom_B(B,Y)\in \Y_T$, hence the property $X\lotimes_R S=X\lotimes_A B\in \X_{T\lotimes_R S}$ can be obtained by applying the adjunction isomorphism.

(4) If $T$ is bounded, then obviously the same property holds true for $T\lotimes_RS$.

(II) The functor $\RHom_{A}(B,-)$ preserves all products, hence it preserves the pure-injectivity property, \cite[Lemma 5.3]{SS20}. Therefore, using the natural isomorphisms 
$$\Hom_{\Der A}(U\lotimes_S S,V)\cong \Hom_{\Der B}(U,\RHom_A(B,V)),$$ a proof can be done by following the lines of the proof for (I). 
\end{proof}

\subsection{The transfer of (co)silting complexes of (co)finite type}  In the following, we will study the transfer of (co)silting complexes of (co)finite type. In this case, we have more connections between the associated TTF triples.

\begin{thm}\label{new-ascend-(co)finite}
Assume that we are in Setting \ref{setting-1}, {and $A$ or $S$ is flat over $R$}. 
\begin{enumerate}[{\rm (I)}]
\item If $C\in\Der A$ is cosilting of cofinite type then $\RHom_R(S,C)$ is cosilting of cofinite type in $\Der B$.
\item  If $T\in \Der A$ is silting of finite type then $T\lotimes_R S$ is of finite type in $\Der B$. 
\end{enumerate}
\end{thm}

\begin{proof}
(I) Since all cosilting objects of cofinite type are pure-injective, it follows that $\RHom_R(S,C)\in\Der B$ is cosilting. 

There exists a set of compact objects $\mathcal{P}\subseteq \U_C$ such that {$\V_C=\mathcal{P}^{\perp_0}$}. Using the fact the the restriction functor preserves coproducts, it follows that $\mathcal{P}\lotimes_R S=\{P\lotimes_R S\mid P\in \mathcal{P}\}$ is a set of compact objects from $\Der B$.

We have $V\in \V_{\RHom_R(S,C)}$ if and only if $V\lotimes_S S=\RHom_S(S,V)\in \V_C$, and this is equivalent to $$0=\Hom_{\Der R}(\mathcal{P}, \RHom_S(S,V))\cong \Hom_{\Der S}(\mathcal{P}\lotimes_R S, V).$$ This implies that {$\V_{\RHom_R(S,C)}=(\mathcal{P}\lotimes_R S)^{\perp_0},$} hence $\RHom_R(S,C)$ is of cofinite type. 

(II) The proof is similar to the proof used for the cosilting case. 
\end{proof}

\begin{lemma}\label{UT}
Suppose that we are in the hypothesis of Setting \ref{setting-1}. Assuming that $\lambda:R\to S$ is faithfully flat, $T\in \Der A$, and $X\in \Der{A^{\mathrm{op}}}$ then 
$X\in{^{\perp_I}(T^+)}$ if and only if $X\lotimes_R S\in{^{\perp_I}[(T\lotimes_RS)^+]}$ (in $\Der{B\op}$).
\end{lemma}

\begin{proof}
Using Lemma \ref{t-str-H}, the equivalence is a consequence of the isomorphisms
$$H^n((T\lotimes_R S)\lotimes_S (X\lotimes_R S))\cong H^n(T\lotimes_R X\lotimes_R S)\cong H^n(T\lotimes_R X)\otimes_R S,$$
that hold for all $X\in \Der R$.
\end{proof}


\begin{prop}\label{YTtoYTS}
Suppose that we are in the hypothesis of Setting \ref{setting-1} and that $\lambda:R\to S$ is faithfully flat. Let $T\in\Der{A}$ be a silting complex of finite type. Using the notations from Remark \ref{TTF-notations}, the following are true:
\begin{enumerate}[{\rm (1)}]
\item $\U_{T^+}=\{X\in \Der{A^{\mathrm{op}}}\mid X\lotimes_R S\in \U_{(T\lotimes_R S)^+}\}$;
\item $\V_{T^+}=\{X\in \Der{A^{\mathrm{op}}}\mid X\lotimes_R S\in \V_{(T\lotimes_R S)^+}\}$;
\item $\Y_{T}=\{Y\in\Der A\mid Y\lotimes_R S\in\Y_{T\lotimes_RS}\}.$
\end{enumerate}
\end{prop}

\begin{proof}
(1) and (2) follow from Lemma \ref{UT}.

(3) The inclusion $\Y_{T}\subseteq\{Y\in\Der A\mid Y\lotimes_R S\in\Y_{T\lotimes_RS}\}$ already appeared in Theorem \ref{ascend}.
 
Conversely, since both $T$ and $T\lotimes_R S$ are of finite type, the associated TTF-triples are compactly generated. Let $Y\in\Der A$ be such that $Y\lotimes_R S\in\Y_{T\lotimes_RS}.$ If $C\in \CX_{T}$ is a compact object then $C\lotimes_R S\in \CX_{T\lotimes_R S}$ by Theorem \ref{ascend}, hence \[\Hom_{\Der{A}}(C\lotimes_R S, Y\lotimes_R S)=0.\] Using the isomorphism presented in Lemma \ref{lemma-zim} we get:
 \[\Hom_{\Der{R}}(C, Y)\otimes_RS=0,\] hence $\Hom_{\Der{R}}(C, Y)=0$, since $S$ is faithfully flat. It follows $Y\in\Y_T$.
\end{proof}

\begin{cor}\label{silt-inj}
Suppose that we are in the hypothesis of Setting \ref{setting-1} and that $\lambda:R\to S$ is faithfully flat. If $T_1,T_2\in \Der A$ are silting objects of finite type such that the $\Der B$ silting objects $T_1\lotimes_R S$ and $T_2\lotimes_R S$ are equivalent then $T_1$ and $T_2$ are equivalent.
\end{cor}

\subsection{Ascend of $n$-tilting modules for $R$-algebras}
We present here generalizations, for $R$-algebras, of some results proved in \cite{HST20} that are useful in our approach. 

If $\CS$ is a family of modules from $\Mod A$ ($A$ is a ring) then  $(\CK_{\CS},\CL_{\CS})$ denotes the cotorsion theory generated by $\CS$ i.e., $\CL=\Ker\prod_{n\geq 1}\Ext_A^n(\CS,-)$. {Recall that an $A$-module is \textit{strongly finitely presented} if it admits a projective resolution consisting of finitely generated projective $A$-modules.}

The proof of the following result follows verbatim from the argument presented in \cite[Proposition 2.3]{HST20}.

\begin{lemma}\label{HST20-2.3} Suppose that we are in the hypothesis of Setting \ref{setting-1} and that $\lambda:R\to S$ is flat.
If $\CS$ is a family of strongly finitely presented modules in $\Mod A$ then $\CK_{\CS}\otimes_R S\subseteq \CK_{\CS\otimes_R S}$ and $\CL_{\CS}\otimes_R S\subseteq \CL_{\CS\otimes_R S}$. 

Moreover, if $\lambda$ is faithfully flat, then
$$\CK_\CS=\{X\in \Mod A\mid X\otimes_R S\in \CK_{\CS\otimes_R S}\},$$ and $$\CL_\CS=\{X\in \Mod A\mid X\otimes_R S\in \CL_{\CS\otimes_R S}\}.$$
\end{lemma}

\begin{prop}\label{tilting-alg}
Suppose that we are in the hypothesis of Setting \ref{setting-1} and that $\lambda:R\to S$ is flat.
If $T$ is an $n$-tilting module and then $T\otimes_R S$ is $n$-tilting. 

Moreover, if $\lambda$ is faithfully flat then 
$$\CK_T=\{X\in \Mod A\mid X\otimes_R S\in \CK_{T\otimes_R S}\},$$ and $$\CL_T=\{X\in \Mod A\mid X\otimes_R S\in \CL_{T\otimes_R S}\}.$$ 
\end{prop}

\begin{proof}
We work in $\Der A$ and replace $T$ by its (bounded) projective resolution. From Lemma \ref{r-comm} it follows that  
$T\lotimes_R S\in T^{\perp_{>0}}$. By Proposition \ref{ascend} we know that $T\lotimes_R S\in (T\lotimes_R S)^{\perp_{>0}}$, and that $T\lotimes_R S$ is a weak generator in $\Der R$. Since $S$ is flat, $T\lotimes_R S=T\otimes_R S$ is also a $B$-module, hence $T\lotimes_R S\in (T\lotimes_R S)^{\perp_{<0}}$. The conclusion follows from Corollary \ref{Corollary 3.6-PoS}.

The last part is a consequence of Lemma \ref{HST20-2.3} since all tilting cotorsion pairs are generated by families of strongly finitely presented modules. 
\end{proof}

\section{Cosilting complexes of cofinite type over commutative rings}\label{Sect-cos-com}

Let $\lambda:R\to S$ be a morphism of commutative rings. In this section we consider the transfers of (co)silting objects of (co)finite type provided by the functors associated to $\lambda$.   

\subsection{Thomason filtrations} Since $R$ is a commutative ring, any compactly generated t-structure $(\U,\V)$ in $\Der R$ is parametrized by so called Thomason filtration on  
$\Spec R$. More precisely we call {\em Thomason (open)} a subset of $\Spec R$
a union of the subsets of the form $V(I)=\{\bp\in\Spec R\mid I\subseteq\bp\}$, 
where $I$ is a finitely generated ideal of $R$ (such a subset is called a {\em basic Thomason set}). Note that the basic Thomason sets are closed in the Zariski topology, that is the Thomason topology is the Hochster dual of the Zariski one.  

A {\em Thomason filtration} on $\Spec R$ is a family $\mathbb{X}=(X_k)_{k\in\Z}$ such that for every integer $k$ we have $X_k\subseteq \Spec R$ is a Thomason set, and $X_{k+1}\subseteq X_k$. This filtration is called {\em non-degenerate} if $\bigcap_n X_n=\varnothing$ and $\bigcup_n X_n=\Spec R$. Moreover, we will say that $\mathbb{X}$ is \textit{bounded} if there exist two integers $m_0$ and $n_0$ such that $X_m=\Spec R$ for all $m\leq m_0$ and $X_n=\emptyset$ for all $n\geq n_0$. It was proved in \cite[Proposition 5.12]{HHZ21} that non-degenerate Thomason filtrations parametrize cosilting complexes of cofinite type. We record, for further use, this correspondence.

\begin{thm}\label{coresp-thomason}
 {\rm (I)} There are bijective correspondences between the following classes:
\begin{enumerate}[{\rm i)}]
\item equivalence classes of cosilting complexes of cofinite type in $\Der R$;
\item compactly generated TTF-triples that are cosuspended and non-degenerate;
\item non-degenerate Thomason filtrations on $\Spec R$.
\end{enumerate}

{\rm (II)}  These bijections restrict to bijections between equivalence classes of bounded cosilting complexes of cofinite type,  TTF-triples that are cointermediate and cosuspended, and bounded Thomason filtrations. 
\end{thm}

\begin{proof} The first part is proved in \cite[Proposition 5.12]{HHZ21}. These bijections extend the bijections described in \cite[Theorem 3.11]{H20}. 

The correspondence between bounded cosilting complexes and cointermediate, cosuspended TTF-triples is provided by \cite[Theorem 3.13]{MV18}, see also  \cite[Corollary 2.14]{AH21}. For the correspondence between TTF-triples that are cointermediate and cosuspended and bounded Thomason filtrations, let $\mathbb{X}=(X_k)_{k\in \Z}$ be a bounded Thomason filtration. By \cite[Theorem 3.11]{H20}, the aisle of the corresponding $t$-structure is $$\U_{\mathbb{X}}=\{X\in \Der R\mid \mathrm{Supp}(H^k(X))\subseteq X_k \textrm{ for all }n\in\Z\}.$$ Since $X_k=\varnothing$ for all $k\geq n$, it follows that for every $X\in \U_{\mathbb{X}}$ we have $H^k(X)=0$ for all $k\geq n$. It follows that $$\Der{R}^{\geq n}=\{X\in\Der R\mid H^i(X)=0 \textrm{ for all } i<n\}\subseteq \U_{\mathbb{X}}^{\perp_0}=\V_{\mathbb{X}}, $$ where $\U_{\mathbb{X}}$ is the coaisle of the $t$-structure associated to $\mathbb{X}$. Conversely, it is easy to see that if $(\U, \V,\W)$ is a TTF-triple that is cointermediate and cosuspended TTF-triple then the associated Thomason filtration is bounded. 
\end{proof}

\subsection{A topological transfer.}

If $\lambda:R\to S$ is a morphism of commutative rings, we can use the topological properties of the associated map $\lambda^\star:\Spec S\to \Spec R$, $\lambda^\star(\bq)=\lambda^{-1}(\bq)$, together with Theorem \ref{coresp-thomason}. Therefore, we obtain a topological method to  transfer the cosilting complexes of cofinite type from $\Der R$ to cosilting complexes of cofinite type in $\Der S$.  We start by recalling some basic results.

\begin{lemma}\label{open-closed}
Let $f:\mathfrak{S}_1\to\mathfrak{S}_2$ be a surjective closed map between two topological spaces. If $f^{-1}(Y)\subseteq \mathfrak{S}_1$ is an open set, then $Y$ is open in $\mathfrak{S}_2$.
\end{lemma}



\begin{lemma}\label{top-prop} {\rm (I)} \cite[Proposition 5.9]{San17} The map $\lambda^\star$ is continuous with respect to Thomason's topology.

{\rm (II)} \cite[Lemma 3.15]{HST20} If $\lambda$ is faithfully flat then $\lambda^\star$ is a surjective closed map with respect to Thomason's topology.
\end{lemma}

We will say that a Thomason subset $X\subseteq\Spec S$ is {\em Thomason saturated} (with respect to $\lambda$) if 
$X=(\lambda^\star)^{-1}(\lambda^\star(X))$. This means that for every $\bq\in\Spec S$ we have $\bq\in X$ if and only if $\lambda^{-1}(\bq)\in\lambda^\star(X)$. 

\begin{thm}\label{topological-transfer} Suppose that $\lambda:R\to S$ is a morphism of commutative rings. If $\mathbb{X}=(X_n)_{n\in \Z}$ is a non-degenerate Thomason filtration in $\Spec R$ then $\mathbb{Y}=((\lambda^\star)^{-1}(X_n))_{n\in\Z}$ is a non-degenerate Thomason filtration in $\Spec S$.

Therefore, $\lambda$ induces a map 
$$\mathbf{\Phi}:\left\{\begin{array}{c}\hbox{compactly generated }\\ \hbox{ t-structures in }\Der{R}\end{array}\right\}\longrightarrow\left\{\begin{array}{c}\hbox{compactly generated }\\ \hbox{ t-structures  in }\Der{S}\end{array}\right\}, $$ that is defined in the following way:

\noindent If $(\U,\V)$ is a compactly generated t-structure in $\Der R$ and $\mathbb{X}=(X_n)_{n\in \Z}$ is the Thomason filtration associated to it then $\mathbf{\Phi}(\U,\V)$ is the $t$-structure in $\Der S$ that is defined by the Thomason filtrations  $\mathbb{Y}=((\lambda^\star)^{-1}(X_n))_{n\in\Z}$. All components of $\mathbb{Y}$ are Thomason saturated with respect to $\lambda$. 

Moreover, if $\lambda$ is faithfully flat then it induces a map $\Psi$ that associates to every compactly generated $t$-structure $(\U',\V')$ in $\Der S$ with the property that all components $Y_n$ of the associated Thomason filtration $\mathbb{Y}=(Y_n)_{n\in \Z}$ are Thomason saturated with respect to $\lambda$, the t-structure of $\Der R$ that corresponds to $\mathbb{X}=(\lambda^\star(Y_n))_{n\in\Z}$. In this case we have $\mathbf{\Psi}\mathbf{\Phi}(\U,\V)=(\U,\V)$ for every compactly generated $t$-structure $(\U,\V)\in \Der R$  
\end{thm}

\begin{proof} Since $\lambda^\star$ is continuous with respect to Hochster's topology it follows that $\mathbb{Y}$ is a system of Thomason sets. 

Observe that $\bq\in (\lambda^\star)^{-1}(X_n)$ if and only if $\lambda^{-1}(\bq)\in X_n$. Using this, it is easy to see that $\mathbb{Y}$ is a Thomason filtration. 

For the last statement, we apply Lemma \ref{top-prop} and Lemma \ref{open-closed} to conclude that the terms of the sequence $\mathbb{X}$ are Thomason sets. As before, it is an easy exercise to prove that $\mathbb{X}$ is a Thomason filtration.   
\end{proof}

\begin{rem}\label{top-transfer-cosilting}
As we already mentioned in the introduction, $\mathbf{\Phi}$ induces a map $$\mathbf{\Phi}:\left\{\begin{array}{c}\hbox{equivalence classes of }\\ \hbox{ cosilting objects }\\ \hbox{of cofinite type in } \Der{R}\end{array}\right\}\longrightarrow\left\{\begin{array}{c}\hbox{equivalence classes of }\\ \hbox{ cosilting objects } \\ \hbox{ of cofinite type in } \Der{S}\end{array}\right\}. $$
If $\lambda$ is faithfully flat, this map is injective.
\end{rem}

\subsection{The ascend of (co)silting complexes of (co)finite type via derived functors} We already proved in Theorem \ref{new-ascend-(co)finite} that the (co)induction functor associated to $\lambda$ preserves the (co)finite type property. In the following we will see that the map induced by the coinduction functor on the equivalence classes of cosilting objects of cofinite type is in fact the map $\Phi$ described in Remark \ref{top-transfer-cosilting}. We will need the following 

\begin{lemma}\label{transfer-T}
Suppose that $\lambda:R\to S$ is a morphism of commutative rings.
\begin{enumerate}[{\rm (I)}]
\item If $C\in\Der R$ is an object such that  $\RHom_R(S,C)\in \Der S$ is cosilting of cofinite type and $\mathbb{Y}=(Y_n)_{n\in\Z}$ is the Thomason filtration associated to $\RHom_R(S,C)$ then all sets $Y_n$ are Thomason saturated with respect to $\lambda$.

\item If $T\in\Der R$ is an object such that  $T\lotimes_R S\in \Der S$ is silting of finite type and $\mathbb{Y}=(Y_n)_{n\in\Z}$ is the Thomason filtration associated to $(T\lotimes_R S)^+$ then all sets $Y_n$ are Thomason saturated with respect to $\lambda$.
\end{enumerate}
\end{lemma}

\begin{proof}
(I) 
Let $n\in \Z$. From \cite[Lemma 3.7]{HHZ21} and Theorem \ref{ascend} {(applied for $A=R$ and, consequently, $B=S$)} it follows that 
\begin{align*}Y_n&=\{\bq\in\Spec S\mid \kappa(\bq)[-n]\in\U_{\RHom_R(S,C)}\}\\ &=\{\bq\in \Spec S\mid \kappa(\bq)[-n]\lotimes_S S\in \U_C\}.\end{align*}

Note that for every $\bq\in \Spec S$ there exists a canonical morphism (of $R$-algebras) $\kappa(\lambda^\star(\bq))\to \kappa(\bq)$ that is induced by the morphism $R/\lambda^{-1}(\bq)\to S/\bq$. This morphism is unital, and $\kappa(\lambda^\star(\bq))$ and $\kappa(\bq)$ are fields. It follows that, as an $R$-module, $\kappa(\bq)$ is a direct sum of copies of $\kappa(\lambda^\star(\bq))$. Then $\Hom_{\Der R}(\kappa(\bq),V)=0$ if and only if $\Hom_{\Der R}(\kappa(\lambda^\star(\bq)),V)=0$. 

Since $\kappa(\bq)[-n]\in \U_{\RHom_R(S,C)}$ if and only if for all $i\leq 0$ we have 
$$0= \Hom_{\Der S}(\kappa(\bq)[-n],\RHom_R(S,C)[i])\cong \Hom_{\Der R}(\kappa(\bq)[-n],C[i]),$$
we conclude that $\kappa(\bq)[-n]\in \U_{\RHom_R(S,C)}$ if and only if for all $i\leq 0$ we have $\Hom_{\Der R}(\kappa(\lambda^\star(\bq))[-n],C[i])=0$. Then $Y_n$ is Thomason saturated with respect to $\lambda$.

(II) For this case we apply (I) and the isomorphism $(T\lotimes_RS)^+\cong \RHom_R(S,T^+)$.
\end{proof}

\begin{thm}\label{ascend-(co)finite}
Let $\lambda:R\to S$ be a morphism of commutative rings. 
\begin{enumerate}[{\rm (I)}]
\item If $C\in\Der R$ is cosilting of cofinite type then  $\RHom_R(S,C)\in\mathbf{\Phi}(C)$, where $\mathbf{\Phi}$ is the map described in Remark \ref{top-transfer-cosilting}.
\item  If $T\in \Der R$ is silting of finite type then $(T\lotimes_R S)^+\in \mathbf{\Phi}(T^+)$. 
\end{enumerate}
\end{thm}

\begin{proof}
(I) From Lemma \ref{transfer-T} it follows that for all $n\in \Z$ we have $Y_n=(\lambda^\star)^{-1}(X_n)$, where $\mathbb{X}=(X_n)_{n\in\Z}$ is the Thomason filtration associated to $C$, and the proof is complete.  


(II) This is also a consequence of Lemma \ref{transfer-T}.
\end{proof}



If $\lambda$ is faithfully flat, we can apply Theorem \ref{ascend-(co)finite} together with the second part of Theorem \ref{topological-transfer}, we obtain the following cosilting version of Corollary \ref{silt-inj}.

\begin{cor}\label{inj-ascend} Suppose that $\lambda:R\to S$ is faithfully flat. If $C_1,C_2\in \Der R$ are cosilting complexes of cofinite type such that the cosilting complexes $\RHom_R(S,C_1)$ and $\RHom_R(S,C_2)$ are equivalent then $C_1$ and $C_2$ are equivalent.
\end{cor}


 

\subsection{Some descent properties} If $\lambda:R\to S$ is faithfully flat, we can use the above results to describe all cosilting complexes of cofinite type from $\Der S$ that can be obtained as images of the functor $\RHom_R(S,-)$.

\begin{prop}
Let $\lambda:R\to S$ be a faithfully flat morphism of commutative rings. If $\overline{C}\in \Der S$ is a cosilting complex of cofinite type and $\mathbb{Y}=(Y_n)_{n\in \Z}$ is the associated Thomason filtration on $\Spec S$, the following are equivalent:
\begin{enumerate}[{\rm (i)}]
\item there exists a cosilting complex of cofinite type $C$ in $\Der R$ such that $\overline{C}$ is equivalent to $\RHom_R(S,C)$;
\item 
all $Y_n$ are Thomason saturated with respect to $\lambda$.  
\end{enumerate}  
\end{prop}

\begin{proof}
i)$\Rightarrow$ii) follows from Lemma \ref{transfer-T}.

ii)$\Rightarrow$i) From the proof of Theorem \ref{topological-transfer} it follows that $\mathbb{X}=(\lambda^\star(Y_n))_{n\in \Z}$ is a Thomason filtration. Let $C\in\Der R$ be a cosilting complex that corresponds to $\mathbb{X}$. It follows from Theorem \ref{ascend-(co)finite} that $\RHom_R(S,C)$ is equivalent to $\overline{C}$.
\end{proof}

Using the equivalence described in Theorem \ref{silt-cosilt-dual}(2), we obtain:

\begin{cor}\label{asc-bound-silt}
Let $\lambda:R\to S$ be a faithfully flat morphism of commutative rings. If $\overline{T}\in \Der S$ is a bounded silting complex, the following are equivalent:
\begin{enumerate}[{\rm (i)}]
\item there exists a (bounded) silting complex $T$ in $\Der R$ such that $\overline{T}$ is equivalent to $T\lotimes_R S$;
\item if $\mathbb{Y}=(Y_n)_{n\in \Z}$ is the Thomason filtration associated to $\overline{T}^+$ then  
all sets $Y_n$ are Thomason saturated with respect to $\lambda$.  
\end{enumerate}  
\end{cor}

\begin{proof}
The implication (i)$\Rightarrow$(ii) is proved in Lemma \ref{transfer-T}. Conversely, since $\mathbb{Y}=(Y_n)_{n\in \Z}$ is a (bounded) Thomason filtration that is saturated with respect to $\lambda$, there exists a bounded cosilting complex $C\in \Der{R}$ such that $\RHom_R(S,C)$ is equivalent to $\overline{T}^+$. Since $C$ is bounded, it follows that $C$ is equivalent to a complex of the form $T^+$ with $T\in \Der R$ a bounded silting complex. Applying  Theorem \ref{silt-cosilt-dual}(2), it follows that the bounded silting complexes $\overline{T}$ and $T\lotimes_R S$ are equivalent.   
\end{proof}

\begin{rem}
The above corollary is also valid for all silting complexes of finite type $\overline{T}\in \Der S$ if we assume that all cosilting complexes of cofinite type from $\Der{R}$ are of the form $T^+$, with $T\in \Der R$ a silting complex of finite type (e.g., if $R$ is noetherian, cf. \cite[Theorem 3.8]{AH21}). 
\end{rem}





The above considerations let us prove a descend property for bounded cosilting complexes:

\begin{thm}\label{a-cosilting}
Let $T\in \Der{R}$ be a bounded complex of projectives. 
If $\RHom(S,T^+)$ is cosilting in $\Der{S}$ then $T^+$ is cosilting in $\Der{R}$. 
\end{thm}

\begin{proof}
Since the Thomason filtration of $\RHom_R(S,T^+)\cong (T\lotimes_R S)^+$ is bounded and Thomason saturated with respect to $\lambda$, we can apply Proposition \ref{asc-bound-silt} to observe that there exists a bounded complex of projectives $\widetilde{T}$ in $\Der{R}$ such that $\widetilde{T}$ is silting and $(\widetilde{T}\lotimes_R S)^+$ is equivalent to $(T\lotimes_R S)^+$. 
It follows that $(\widetilde{T}\lotimes_R S)^+$ and $(T\lotimes_R S)^+$, induce the same t-structure whose aisle and coaisle are  
\[\V_{(T\lotimes_RS)^+}={^{\perp_{\leq0}}[(T\lotimes_RS)^+]}={^{\perp_{\leq0}}[(\widetilde{T}\lotimes_RS)^+]},\] respectively \[\W_{(T\lotimes_RS)^+}={^{\perp_{>0}}[(T\lotimes_RS)^+]}={^{\perp_{>0}}[(\widetilde{T}\lotimes_RS)^+]}.\]
Using Corollary \ref{UT} twice, we get 
\[{^{\perp_{\leq0}}(T^+)}=\{X\in\Der R\mid X\lotimes_RS\in\V_{(T\lotimes_RS)^+}\}={^{\perp_{\leq0}}(\widetilde{T}^+)}\]
and 
\[{^{\perp_{>0}}(T^+)}=\{X\in\Der R\mid X\lotimes_RS\in\W_{(T\lotimes_RS)^+}\}={^{\perp_{>0}}(\widetilde{T}^+)}.\]
Therefore the pair $({^{\perp_{\leq0}}(T^+)}, {^{\perp_{>0}}(T^+)})$ coincides with the cosilting t-structure induced by $(\widetilde T)^+$. Hence $T^+$ is cosilting too and it is obviously equivalent to $(\widetilde T)^+$.
\end{proof}

\section{The descend property for bounded silting complexes}

\subsection{Bounded silting complexes as modules over Dynkin quivers}

Let $T$ be a bounded silting complex in $\Der R$. We assume that it is concentrated in the degrees $-n+1,\dots,-1,0$, and we will use the interpretation presented in 
\cite{MV18}. In order to do this, we consider the Dynkin quiver 
$$\mathbb{A}_n: \overset{-n+1} \bullet \longrightarrow \overset{-n+2} \bullet \longrightarrow \dots \longrightarrow \overset{-1} \bullet \longrightarrow \overset{0} \bullet  ,$$
and we denote by $A(R)$ the $R$-algebra $R\mathbb{A}_n/I$, where I is the ideal generated by all paths of length $2$. Note that $A(R)\cong T_n(R)/J(T_n(R))^2$, where $T_n(R)$ denotes the ring of lower triangular matrices over $R$ and $J(T_n(R))$ is the corresponding Jacobson radical. 

Let $\mathrm{Rep}(A(R))$ be the category of representations bounded by $I$ of $\mathbb{A}_n$ in $\Mod R$. Then $\mathrm{Rep}(A(R))$ is equivalent to the category  $\Mod A(R)$, and it can be identified to the category of complexes over $R$ that are concentrated in $-n+1,\dots,-1,0$, so we have a fully faithful functor from $\mathrm{Rep}(A(R))$ to the category of complexes. If $X$ is a complex concentrated in $-n+1,\dots,-1,0$, we will denote by $\widehat{X}$ the corresponding object from $\mathrm{Rep}(A(R))$. The above mentioned functor induces a functor $\mathbf{\Psi}_R:\mathrm{Rep}(A(R))\to \Der R$. We have the following useful result

\begin{lemma}\cite[Lemma 2.4]{MV18}\label{ext-MV} If $\mathbf{\Psi}(\widehat{T})$ is a complex of projectives then for every $\widehat{X}\in \mathrm{Rep}(A(R))$ and for every $i>0$ we have 
$$\Hom_{\Der R}(\mathbf{\Psi}_R(\widehat{T}), \mathbf{\Psi}_R(\widehat{X})[i])=\Ext^i_{\mathrm{Rep}(A(R))}(\widehat{T},\widehat{X}).$$
\end{lemma}

This correspondence is compatible with respect to change of rings. 

\begin{lemma}
Suppose that $\lambda:R\to S$ is a flat morphism of commutative ring. Then $A(S)\cong A(R)\otimes S$ and we have a commutative diagram
$$\xymatrix{\mathrm{Rep}(A(R))\ar[rr]^{-\otimes_R S}\ar[d]_{\mathbf{\Psi}_R} && \mathrm{Rep}(A(S))\ar[d]^{\mathbf{\Psi}_S} \\
\Der R \ar[rr]^{-\lotimes_R S} && \Der S
}
$$
\end{lemma}

\begin{proof}
For the first isomorphism we observe that $(R\mathbb{A}_n/I)\otimes_R S\cong (R\mathbb{A}_n\otimes_R S)/(I\otimes_R S)\cong S\mathbb{A}_n/K$, where $K$ is the ideal in $S\mathbb{A}_n$ generated by all paths of length $2$ (see \cite{Con}). The commutativity of the diagram can be checked by direct computations.
\end{proof}

With these preparatory considerations in hand, we are ready to prove the following 

\begin{prop}\label{desc-self-ext-orthogonal}
Let $\lambda:R\to S$ be a faithfully flat morphism of commutative rings.
Suppose that $T$ is a bounded complex of projective $R$-modules that is concentrated in $-n+1,\dots,-1,0$. If $T\lotimes_R S$ is silting then $\Add(T)\subseteq T^{\perp_{>0}}$.  
\end{prop}

\begin{proof}
Since $T\otimes_R S$ is bounded silting, we can use Lemma \ref{transfer-T} and Corollary \ref{asc-bound-silt} to observe that there exists a silting complex $P\in D(R)$ such that $P\lotimes_R S$ is equivalent to $T\lotimes_R S$. Moreover, since $\lambda$ is faithfully flat, we can assume that $P$ is concentrated in $-n+1,\dots,-1,0$. We use \cite[Theorem 2.10]{MV18} to  observe that $\widehat{P}$ is an $(n-1)$-tilting module over $A(R)$, and that $\widehat{P\lotimes_R S}\cong \widehat{P}\otimes_R S$ and $\widehat{T\lotimes_R S}$ are equivalent $(n-1)$-tilting modules over $A(S)$. From Proposition \ref{tilting-alg} it follows that $\widehat{T}^{(I)}\in \CK_{\widehat{P}}\bigcap \CL_{\widehat{P}}$ for all sets $I$. Then $\Ext^{i}_{A(R)}(T,\Add(T))=0$ for all $i>0$. Since $T$ is a complex of projectives, we can use Lemma \ref{ext-MV} to obtain the conclusion. 
\end{proof}
 
 There, we now see that the condition (Sb1) of the characterization of bounded silting complexes from Theorem~\ref{char-silt-bounded} descends along all faithfully flat morphisms. Although we are not able to show an analogous result for (Sb2) in full generality, we shall show that it holds in many situations.  
 
 Recall that a full subcategory $\C$ of $\Der R$ is \textit{localizing} if it is a triangulated subcategory closed under coproducts. If $\C$ is an arbitrary subcategory of $\Der R$ we denote by $\Loc(\C)$ the smallest localizing subcategory of $\Der R$ containing $\C$, and we write just $\Loc(X)$ if $\C = \{X\}$. Recall from \cite[Theorem 3.14]{An-19} that $(\Loc(X),X^{\perp_\Z})$ is a t-structure. Therefore,  $\Loc(X) = \Der R$ if and only if $X$ is a generator of $\Der R$, that is, $X^{\perp_{\Z}} = 0$. We will freely use the fact that any localizing subcategory $\C$ of $\Der R$ is a tensor ideal, meaning that $X \lotimes_R C \in \C$ for any $C \in \C$ and $X \in \Der R$, see \cite[Lemma 1.1.8]{KP}.
 
 \begin{prop}\label{descent-ff-gen}
Let $R$ be a commutative ring and $R \to S$ a faithfully flat homomorphism of commutative rings such that $\Loc(S) = \Der R$. 

Suppose that $T\in \Der R$ is isomorphic to a bounded complex of projective $R$-modules. If $T\lotimes_R S$ is silting in $\Der S$ then $T$ is silting in $\Der R$.
\end{prop}
 \begin{proof}
By Proposition~\ref{desc-self-ext-orthogonal}, we already know that (Sb1) from Theorem~\ref{char-silt-bounded} holds for $T$. It remains to show that (Sb2), that is, we need to check $\Loc(T) = \Der R$. We have $T \lotimes_R S \in \Loc(T)$ and  by the assumption, $T \lotimes_R S$ is a silting in $\Der S$. It follows that $S \in \Loc(T)$ (this is clear since $T \lotimes_R S$ satisfies condition (S2) of \cite[Definition 5.1]{An-19} thanks to \cite[Proposition 5.3]{An-19}), and therefore $\Loc(T) = \Der R$ by the assumption we made on $S$.
\end{proof}

Using Corollary~\ref{Corollary 3.6-PoS} we obtain 

\begin{cor}\label{tilt-descent-ff-gen}
Let $R$ be a commutative ring and $R \to S$ a faithfully flat homomorphism of commutative rings such that $\Loc(S) = \Der R$. 
If $T\in \ModR$ is a module such that $T\otimes_R S\in \Mod S$ is $n$-tilting then $T$ is $n$-tilting. 
\end{cor}

\begin{rem}
Following \cite{PS18}, a silting object $T$ is called a \emph{tilting object} if in addition we have $\Add(T) \subseteq T^{\perp_{<0}}$. If $T \in \Der R$ is bounded silting, then $T$ is tilting precisely if the realization functor from the bounded derived category of the heart of the silting t-structure to the bounded derived category of $R$ is an equivalence \cite[Corollary 5.2]{PS18}. Any $n$-tilting module is a tilting object as an object of the derived category.
 
 We don't know if the descent (or even the ascent) properties we established for bounded silting complexes restrict to tilting complexes, in general. In particular, our approach which used the transfer to the category $\mathrm{Rep}(A(R))$ does not offer information about the groups $\Hom(T, T^{(I)}[k])$ for $k<0$. 
 \end{rem}
 
 An important example of a faithfully flat morphism is the so-called \textit{Zariski cover}, that is, a ring morphism of the form $R \to \prod_{i=0}^{n-1} R[f_i^{-1}]$ for some finite generating set $f_0,f_1,\ldots,f_{n-1}$ of the regular module $R$, also see \S\ref{subsec-locality}.

 \begin{lemma}\label{zariski-cover}
Let $R \to S$ be a Zariski cover of $R$, then $\Loc(S) = \Der R$.
 \end{lemma}

 \begin{proof}
    This is essentially proved in \cite[Lemma 4.1]{HST20}.
 \end{proof}
 For the definition of a local property and an ad-property, see \S\ref{subsec-locality}.
 \begin{lemma}\label{bounded-local}
    The property of ${X} \in \Der R$ of being isomorphic to a bounded complex of projectives is an ad-property.
 \end{lemma}
 \begin{proof}
    Clearly, if $X$ is a bounded complex of projective $R$-modules then $X \otimes_R S$ is a bounded complex of projective $S$-modules for any flat ring morphism $R \to S$.
 
    Let $R \to S$ be faithfully flat and assume without loss of generality that $X \otimes_R S$ is a complex of projective $S$-modules concentrated in degrees $-n,-n+1,\ldots,0$. Since $R \to S$ is faithfully flat and $X \otimes_R S$ is bounded, we see immediately that homology of $X$ vanishes outside of degrees $-n,-n+1,\ldots,0$. Therefore, there is a complex $P$ of projective $R$-modules concentrated in non-positive degrees which is isomorphic to $X$ in $\Der R$. Since $P \otimes_R S$ is a complex of projective $S$-modules isomorphic in $\Der S$ to $X \otimes_R S$. By the assumption on $X \otimes_R S$, the cokernel of the differential map $X^{-n-1} \otimes_R S \to X^{-n} \otimes_R S$, which is the same as $\Coker(X^{-n-1} \to X^{-n}) \otimes_R S$, is a projective $S$-module. Then by the result of Raynaud and Gruson \cite[Theorem 95.5]{stacks}, the cokernel of the differential $X^{-n-1} \to X^{-n}$ is a projective $R$-module. It follows that the soft truncation $\tau^{\geq -n}X$, which is quasi-isomorphic to $X$, is a bounded complex of projective $R$-modules.
 \end{proof}
 \begin{cor}\label{silt-zar-loc}
The property of an object $T \in \Der R$ being bounded silting is (Zariski) local.
 \end{cor}
 \begin{proof}
  Combine Theorem~\ref{ascend}, Lemma~\ref{zariski-cover} and Lemma~\ref{bounded-local},  and apply the Affine Communication Lemma \ref{affine-communication}.
 \end{proof}
 
 Denote by $\I$ the full subcategory of all pure-injective objects of $\Der R$.
 \begin{thm}\label{loc-pure-inj}
    Let $R$ be a commutative ring with the property that $\Loc(\I) = \Der R$. Then any faithfully flat ring homomorphism $R \to S$ of commutative rings has the property that $\Loc(S) = \Der R$.
 \end{thm}
 \begin{proof}
    Recall from \cite[Lemma 35.4.8.]{stacks} that the map $\lambda: R \to S$ is a pure monomorphism of $R$-modules. It follows that for any object $X \in \Der R$ the morphism $X \lotimes_R \lambda: X \to X \lotimes_R S$ is a pure monomorphism in $\Der R$ by \cite[Lemma 2.6]{AH21}. If $X$ is pure-injective, then $X \lotimes_R \lambda$ is a split monomorphism and so $X$ is a direct summand in $X \lotimes_R S$. Since $X \lotimes_R S \in \Loc(S)$, we have $X \in \Loc(S)$. Then $\I \subseteq \Loc(S)$ and so $\Loc(S) = \Der R$.
 \end{proof}
 \begin{rem}\label{list-descend}
 The condition $\Loc(\I) = \Der R$ holds in the following hypotheses:
    \begin{itemize}
        \item $R$ noetherian, this follows from \cite[Theorem 3.3]{Rick},
        \item $R$ admits a finite injective resolution, in particular, if $R$ has finite global dimension,
        \item more generally, if $R$ is isomorphic to a bounded complex of pure-injectives, in particular, if $R$ has finite pure global dimension, in particular, if $R$ is countable or of cardinality $\aleph_n$ for some $n > 0$ \cite[Proposition 10.5]{GJ}, 
        \item if $\Der R$ is a compactly generated triangulated category of finite pure global dimension in the sense of \cite{Bel}, in particular, if $\Der R$ satisfies the ``Brown representability for homology of morphisms``, see \cite{CKN}.
        \item if $R$ is such that every localizing subcategory of $\Der R$ is cohomological in the sense of Krause, see e.g. \cite[\S 3.2]{KS}.
    \end{itemize}
    We also remark that $\Loc(S) = \Der R$ where $R \to S$ is a faithfully flat homomorphism of commutative rings such that $S$ is of projective dimension at most $1$ over $R$ and such that the projective dimension of flat $R$-modules are bounded by some positive integer is proved in \cite[Proposition 4.2]{CI}.
\end{rem} 

\begin{rem}
Combining the above, we see that if $\Loc(\I) = \Der R$ holds for {\em all} commutative rings then the property of being a bounded silting object is even an ad-property.
We do not know any ring for which $\Loc(\I) = \Der R$ fails. An example for which injective modules do not generate $\Der R$ is in \cite{Rick}. If $R$ is von Neumann regular that $\Loc(\I) = \Der R$ precisely if injectives generate in the sense of \cite{Rick}.
 \end{rem}
 
 \section*{Acknowledgements}
The research of S.\ Breaz and G.-C. Modoi is supported by a grant of the Ministry of Research, Innovation and Digitization, CNCS/CCCDI--UEFISCDI, project number PN-III-P4-ID-PCE-2020-0454, within PNCDI III. 

The research of M. Hrbek is supported by the GAČR project 20-13778S and RVO: 67985840.

\end{document}